\documentclass[reqno]{amsart}
\usepackage[usenames, dvipsnames]{color}
\usepackage{amsthm}
\usepackage{amsmath}
\usepackage{amssymb}
\usepackage{amscd}
\usepackage{graphics}
\usepackage{latexsym}
\usepackage{stmaryrd}
\usepackage{hyperref}
\usepackage{empheq}

\theoremstyle{plain}
\newtheorem{theorem}{Theorem}[section]
\newtheorem{thm}[theorem]{Theorem}
\newtheorem{cor}[theorem]{Corollary}
\newtheorem{lem}[theorem]{Lemma}
\newtheorem{prop}[theorem]{Proposition}

\newtheorem{conj}[theorem]{Conjecture}

\theoremstyle{definition}
\newtheorem{defn}[theorem]{Definition}
\newtheorem{prob}[theorem]{Problem}

\newtheorem{rmk}[theorem]{Remark}

\theoremstyle{remark}

\definecolor{titlecol}{named}{BrickRed}
\definecolor{headcol}{named}{Violet}
\definecolor{seccol}{named}{Red}
\definecolor{sseccol}{named}{Bittersweet}
\definecolor{pbcol}{named}{Black}
\definecolor{sncol}{named}{Brown}
\definecolor{acol1}{named}{Red}
\definecolor{acol2}{named}{Apricot}

\newcommand\cE{{\mathcal E}}

\def\tr{\mathop{\rm tr}\nolimits}

\def\cE{{\mathcal E}}

\def\cK{{\mathcal K}}

\def\cH{{\mathcal H}}

\def\p{\partial}

\begin{document}

\title{Constant scalar curvature equation and regularity of its weak solution}
\date{\today}

\author{Weiyong He}
\address{Department of Mathematics, University of Oregon, Eugene, OR 97403. }
\email{whe@uoregon.edu}

\author{Yu Zeng}
\address{Department of Mathematics, University of Rochester, Rochester, NY 14627.}
\email{yu.zeng@rochester.edu}

\begin{abstract}In this paper we study constant scalar curvature equation (CSCK), a nonlinear fourth order elliptic equation, and its weak solutions on K\"ahler manifolds. We first define a notion of weak solution of CSCK for an $L^\infty$ K\"ahler metric. The main result is to show that such a weak solution (with uniform $L^\infty$ bound) is smooth. As an application, this answers in part a conjecture of Chen regarding the regularity of $K$-energy minimizers.  The new technical ingredient is a $W^{2, 2}$ regularity result for the Laplacian equation $\Delta_g u=f$ on K\"ahler manifolds, where the metric has only $L^\infty$ coefficients. It is well-known that such a $W^{2, 2}$ regularity ($W^{2, p}$ regularity for any $p>1$) fails in general (except for dimension two) for uniform elliptic equations of the form $a^{ij}\p^2_{ij}u=f$ for $a^{ij}\in L^\infty$, without certain smallness assumptions on the local oscillation of $a^{ij}$. We observe that the K\"ahler condition plays an essential role to obtain a $W^{2, 2}$ regularity for elliptic equations with only $L^\infty$ elliptic coefficients on compact manifolds.\end{abstract}
\maketitle

\section{Introduction}

Let $(M, [\omega], J)$ be a compact K\"ahler manifold with a fixed K\"ahler class. In 1980s Calabi  proposed to find a canonical representative of of K\"ahler metrics in $[\omega]$ in his seminal work \cite{C1, C2}, by minimizing the $L^2$ energy of curvature. A critical point, called \emph{extremal} metric (extK), satisfies the equation that $\nabla^{1, 0}R$ is a holomorphic vector field, where $R$ is the scalar curvature. In particular a K\"ahler metric with constant scalar curvature (CSCK) is an extremal metric.  A fundamental problem in K\"ahler geometry is the existence of a CSCK (or extK) in a given K\"ahler class. Examples of CSCK contain K\"ahler-Einstein metrics when the first Chern class $c_1$ is definite or zero. When $c_1=0$ or $c_1<0$, there always exists a K\"ahler-Einstein metrics by the famous work of Yau \cite{Yau} and Aubin \cite{Aubin} (for $c_1<0$). In general there are obstructions to the existence of CSCK. When the K\"ahler class is integral, i.e. $2\pi[\omega]$ is given by the first Chern class $c_1(L)$ of a positive line bundle $L\rightarrow M$ ($M$ is projective by the famous Kodaira embedding),  the so-called Yau-Tian-Donaldson conjecture predicts that $c_1(L)$ contains a CSCK if and only if the polarized manifold $(M, L)$ is K-stable. When the anti-canonical bundle is positive ($c_1>0$), namely $M$ is Fano, recently Chen, Donaldson and Sun have proved that a Fano manifold supports a K\"ahler-Einstein metric if and only if it is K-stable (\cite{CDS1}, \cite{CDS2} and \cite{CDS3}). Donaldson \cite{Donaldson} has proved that the Yau-Tian-Donaldson conjecture holds for toric K\"ahler surfaces. There are also vast existence results of CSCK (or extK) on concrete examples. On one hand,
when a CSCK exists on a projective manifold, then it is K-stable, see \cite{Donaldson02, CT, Stoppa, Mabuchi, BDL}; but on the other hand the general existence results remain open. 
 We shall emphasize the study of CSCK is a huge subject and there are vast research in literature of the subject.  
 
 In this paper we focus on its PDE aspect, to consider the regularity of  constant scalar curvature equation in weak sense for $L^\infty$ K\"ahler metrics. 

All K\"ahler metrics in $[\omega]$ can be parametrized by smooth functions in the space of K\"ahler potentials. 
\[
\cH_\omega=\{\phi\in  C^\infty(M): \omega_\phi=\omega+\sqrt{-1}\p\bar\p \phi>0\}
\]
For any $\phi\in \cH_\omega$, the scalar curvature is given by
\[
R_\phi=-g^{i\bar j}_\phi \p_i\p_{\bar j}\log (\det(g_{i\bar j}+\phi_{i\bar j}))
\]
and the constant scalar curvature equation is simply given by,
\begin{equation}\label{CSCK}
R_\phi=\underline{R}
\end{equation}
where $\underline{R}$ is a constant depending only on $(M, [\omega], J)$. Mabuchi \cite{M1} introduced the \emph{K-energy} $\cK$ (or the \emph{Mabuchi functional}), a functional defined on $\cH_\omega$, characterized by its variation
\begin{equation}\label{kenergy}
\delta \cK=-\int_M (\delta\phi)(R_\phi-\underline{R}) \omega_\phi^n.
\end{equation}
Mabuchi \cite{M2} also introduced a natural Riemannian metric on $\cH_\omega$, with the norm of a tangent vector $v\in C^\infty(M)$ at $\phi\in \cH_\omega$ defined by
\[
\|v\|^2=\int_M v^2\omega^n_\phi
\]
The $\cK$-energy is a \emph{convex} functional on $\cH_\omega$ with respect to the Mabuchi metric, i.e along (smooth) geodesics. Later on Donaldson \cite{Donaldson97} described a beautiful geometric picture of  $\cH_\omega$ with the Mabuchi metric, and set up a program which tights up the $\cK$-energy, the existence and uniqueness of CSCK and the notion of stability in Geometric Invariant Theory (GIT). The geodesic equation in $\cH_\omega$ can be written as a homogeneous complex Monge-Ampere equation, observed by Semmes \cite{Semmes} and Donaldson \cite{Donaldson97}. Chen \cite{Chen} confirmed a conjecture of Donaldson by proving that given any two points in $\cH_\omega$, there exists a unique solution to the homogenous complex-Ampere equation with $C^{1, \bar 1}$ regularity, which are potentials of a K\"ahler current in $[\omega]$ with $L^\infty$ coefficients. This leads (naturally) to the following conjecture of Chen, 

\begin{conj}[Chen]\label{Chen}A $C^{1, \bar 1}$ minimizer in a given K\"ahler class of the K-energy is a smooth K\"ahler metric with constant scalar curvature. 
\end{conj}

A point to make is that the definition of $\cK$ in \eqref{kenergy} requires regularity for $\phi\in C^4$ and $\omega_\phi$ is strictly positive, while Chen \cite{chen00} observed that $\cK$ has a nice well-defined formula for $C^{1, \bar 1}$ potentials.  

A CSCK is evidently a critical point of $\cK$ and indeed, it is a minimizer of $\cK$, proved by Donaldson \cite{Donaldson01} and Li \cite{CL} for integral classes via quantization strategy and Chen-Tian \cite{CT, CT2} for general cases via the convexity of $\cK$ along $C^{1, \bar 1}$ geodesics.  We should mention that even though formally $\cK$ is convex but due to the lack of regularity beyond $C^{1, \bar 1}$, it is a very subtle and deep problem to prove $\cK$ is convex along $C^{1, \bar 1}$ geodesics. Very recently Berman-Berndtsson \cite{BB} proved that $\cK$ is convex along these weak geodesics (see also Chen-Li-Paun \cite{CLP}). In this perspective, it is natural and intuitive to approach the existence of CSCK by looking for a minimizer of K-energy. Regularity of such a minimizer is then an outstanding problem.  

Recently Darvas \cite{Darvas14} has provided a deep understanding of the metric completion of the space of K\"ahler potentials $\cH_\omega$ with the Mabuchi metric.  He proved that the metric completion of $\cH_\omega$ with respect to the Mabuchi metric is given by an energy class $\cE^2$ of $\omega$-plurisubharmonic functions, with the $d_2$ metric. He also realized that the energy class $\cE^1$, which is the metric completion of $\cH_\omega$ with the $d_1$ metric is also important. These energy classes $\cE^p$
were introduced and studied by Guedj-Zeriahi \cite{GZ}. 
Building on Darvar's work, Darvas-Rubinstein \cite{DR} proposed to study the minimizers of $\cK$ over $\cE^1$ and conjectured that a minimizer of $\cK$ over $\cE^1$ is a smooth K\"ahler metric with constant scalar curvature. Their proposal incorporates naturally the properness conjecture of Tian \cite{tian00}, that the existence of CSCK is equivalent to the properness of K-energy and a conjecture of Chen \cite{chen08}, that using the geometry of $\cH_\omega$, in particular the distance of $\cH_\omega$ to define the properness of K-energy.  With such a properness condition, Darvas-Rubinstein proved that there always exists a minimizer of $\cK$ in $\cE^1$. A prominent problem is then to study the regularity of such a minimizer. However in this paper we will only focus on Chen's conjecture from the PDE perspective. While the more general conjecture made by Darvas-Rubinstein (the regularity theory in such general sense) seems to be out of reach at the moment. We should mention that recently Berman-Darvas-Lu  \cite{BDL} proved that a minimizer of K-energy in $\cE^1$ is indeed a smooth CSCK if there exists a CSCK. Their results can be viewed as a generalization of uniqueness of CSCK to the class of $\cE^1$. 

Now we come back to Conjecture \ref{Chen}. When the K\"ahler class is proportional to the canonical class (hence $c_1$ is either positive, zero or negative), this conjecture was verified by Chen-Tian-Zhang(\cite{CTZ}) using the K\"ahler-Ricci flow with rough initial data. More generally, Berman \cite{Berman} proved that in this case, a $K$-energy minimizer in $\cE^1$ is indeed a smooth K\"ahler-Einstein. In both \cite{CTZ} and \cite{Berman}, it is essential that the first Chern class is definite or zero, and the K-energy minimizer satisfies a second order elliptic equation (in weak sense) in nature. 
For a general K\"ahler class Chen's conjecture remains largely open. There are two essential difficulties. 
The first one is that CSCK is of fourth order in nature, with the nonlinearity given by the complex Monge-Ampere operator composed by a metric Laplacian.
The second one is that the potential is only assumed to be K\"ahler current, hence the strict positiveness is missing. In this case it is a tricky question even to have a PDE to work on. In this paper we will study the regularity of a minimizer which is also an $L^\infty$ K\"ahler metric, tackling the first difficulty. 
Hence we assume in addition a $C^{1, \bar 1}$ minimizer satisfies a strict positive condition in $L^\infty$ sense. This assumption allows us to introduce a notion of a weak solution for $R_\phi=\underline{R}$, for an $L^\infty$ K\"ahler metric $\omega_\phi$, using the PDE theory. We write the equation as
\[
R_\phi-\underline{R}=-\Delta_\phi \left(\log \frac{\det(g_{i\bar j}+\phi_{i\bar j})}{\det{g_{i\bar j}}}\right)+g^{i\bar j}_\phi R_{i\bar j}-\underline{R}=0.
\]
This allows us to interpret a weak solution of CSCK for $L^\infty$ K\"ahler metric as the following, that the volume ratio $V_\phi=\log \frac{\det(g_{i\bar j}+\phi_{i\bar j})}{\det{g_{i\bar j}}}$ is a weak solution of 
\[
\Delta_\phi V_\phi=g^{i\bar j}_\phi R_{i\bar j}-\underline{R}
\]
in the distribution sense (see Section 2 for the details).  Our main result is the following regularity for such a weak solution.

\begin{thm}\label{main0}Let $\omega_\phi$ be an $L^\infty$ K\"ahler metric in the class $[\omega]$ such that 
\[
\epsilon \omega\leq \omega_\phi\leq \Lambda \omega,
\]
holds in $L^\infty$ sense with two positive constants $\epsilon<\Lambda$. If $\omega_\phi$ is a weak solution of $R_\phi=\underline{R}$, then $\omega_\phi$ is smooth. Moreover, there are uniform positive constants $c_1=c_1(n, \Lambda)$ and $C=C(n, k, \Lambda)$ such that
\[
c_1 \omega\leq \omega_\phi\leq \Lambda \omega,\; \|\omega_\phi\|_{C^k}\leq C(k, n, \Lambda). 
\]
We emphasize that these estimates are independent of $\epsilon$. 
\end{thm}

As a direct application of this regularity, we confirm partly Chen's conjecture. 
\begin{thm}\label{main}If a uniformly $L^\infty$ K\"ahler metric minimizes $K$-energy, then it is a smooth CSCK. 
\end{thm}
We can also prove similar results for modified $K$-energy, with slight modification of the arguments, that an invariant K\"ahler metric which minimizes modified $K$-energy is smooth, if it is uniformly $L^\infty$.

Previously the authors \cite{HZ} studied the Calabi flow with rough initial data, and proved that the Calabi flow has a unique smooth solution for an $L^\infty$ K\"ahler metric with small oscillation. As an application, we proved Chen's conjecture for a K-energy minimizer if it is an $L^\infty$ metric with small oscillation; in particular, this asserts that if a continuous K\"ahler metrics minimizes K-energy, then it is a smooth CSCK. However, certain smallness assumption seems to be essential for the approach in \cite{HZ}. 
On the other hand, recently Chen-Warren \cite{CW} studied the hamiltonian stationary equation for Lagrangian submanifolds, which is also a fourth order equation, and they proved regularity and removable singularity theorem for such an equation. These together motivate us to study  a regularity and removable singularity theorem for weak solutions of CSCK for $L^\infty$ metrics with certain smallness assumptions.  Indeed we have similar local regularity results as in Chen-Warren \cite{CW} (for hamiltonian stationary equation for Lagrangian submanifolds), with suitable smallness assumptions. 

\begin{thm}Let $Q\subset B=B(1)$ be a compact set with capacity zero. Then there exists $c_n$ such that $u\in C^{1, \bar 1}(B\backslash Q)$ is a weak solution of $R_u=f$ in the sense of \eqref{weak1} in $B\backslash Q$ for some smooth function $f$ such that  \[
(1-c_n)\delta_{i \bar j}\leq u_{i\bar j}(t)\leq (1+c_n)\delta_{i \bar j},
\]
then $u$ is a smooth solution of in $B$.
\end{thm}

This regularity result is local in nature and we have a direct corollary, 
\begin{cor}[\cite{HZ}]\label{reg1}Given $(M, g)$ and suppose $\psi\in \cH^{1, 1}_{w+}$ is a weak solution of $R_u=\underline{R}$.  There exists a positive dimensional constant $c_0$ such that if  $\|\p\bar \p \psi\|_g\leq c_0$, then $\psi$ is smooth. 
\end{cor}

The metric $\omega_\psi$ defines an $L^\infty$ K\"ahler metric with $\psi\in \cH^{1, 1}_{w+}$, see Section 2 for precise definitions. Corollary \ref{reg1} was recently proved by the authors, as an application of  the Calabi flow with rough initial data. Note that we do not have any addition geometric assumption on the background metric $g$, except its smoothness. However, it seems to be a hard and subtle problem to prove the local regularity if we do not have any smallness assumption. The essential difficulty is related to the following fact. 
The Caldron-Zygmund theory says if $a^{ij}$ is continuous, then  $W^{2, p}$ regularity holds for the uniform elliptic equation, with $f\in L^p$ for any $p\geq 1$
\[a^{ij}u_{ij}=f. 
\]
The Caldron-Zygmund theory depends on the continuous module of $a^{ij}$. As a comparison, there is no $W^{2, p}$ estimate in general  for $a^{ij}$ uniformly elliptic with only $L^\infty$ coefficients, with a counterexample by Pucci (see \cite{PT}).  We should mention there are generalizations of Caldron-Zygmund theory with weaker assumptions in $a^{ij}$, for example for $a^{ij}\in L^\infty\cap VMO$, Chiarenza-Fascal-Longo \cite{CFL} proved such a $W^{2, p}$ estimates; but clearly certain smallness condition is necessary. 

Hence we need a totally new approach for Theorem \ref{main0}. Our observation to overcome this difficulty, hence to obtain Theorem \ref{main0}, is to realize that the K\"ahler condition plays a very important role. In particular, for smooth K\"ahler metric, we always have
\begin{equation}\label{kcondition}
\p_{a}\left(g^{i\bar j}\det(g_{i\bar j})\right)=0,\; \text{for}\; a=i, \bar j
\end{equation}
The K\"ahler condition essentially allows us to treat the linear equation of the non-divergence form
\begin{equation}\label{w1}
\Delta_g u=g^{i\bar j} u_{i\bar j}=f
\end{equation}
and the divergence form 
\begin{equation}\label{w2}
\Delta_g u=\det(g_{i\bar j})^{-1}\p_i \left(g^{i\bar j} u_{\bar j} \det(g_{i\bar j})\right)=f,
\end{equation}
together even when $g^{i\bar j}$ is merely in $L^\infty$. This observation allows us to build a bridge between (weak) solutions of two equations \eqref{w1} and \eqref{w2}. In particular, we prove that on a compact K\"ahler manifold, we have the desired $W^{2, 2}$ estimate for a weak solution of \eqref{w2} and \eqref{w1}, with only $L^\infty$ coefficients. 

\begin{thm}\label{weak}Let $g=(g_{i\bar j})$ be an $L^\infty$ K\"ahler metric. For $f\in L^2(M)$ such that $\int_M f \omega^n=0$, then there exists a unique solution $W^{2, 2}$ (up to addition of constants) of \eqref{w1} and \eqref{w2}.\end{thm}

The standard elliptic theory gives a (unique) $W^{1, 2}$ solution of \eqref{w2}; hence the new part of Theorem \ref{weak} is the $W^{2, 2}$ regularity and $u$ also solves \eqref{w1}. 
Using this regularity of linear equations, we are able to prove a uniqueness theorem for weak solutions of linear equation in distribution sense. We can derive that the volume ratio is indeed $W^{1, 2}$ (and hence $W^{2, 2}\cap C^\alpha$), for a weak solution defined in distribution sense, relying crucially on this uniqueness result. This eventually leads to the proof of  our main regularity Theorem \ref{main0}, with only $L^\infty$ assumption. Another interesting feature is that these estimates are independent of $\epsilon$, the lower bound of $\omega_\phi$.  (Even though we need a strictly positive lower bound of $\omega_\phi$ in $L^\infty$ sense to define a weak solution.) These estimates would be proven by a priori estimates considered in Section 6, where we consider also twisted CSCK. 

In Section 6 we address the difficulty of strict positivity assumption. We study a priori estimate of CSCK and twisted CSCK. Twisted CSCK appears in the work of J. Fine \cite{Fine} (see also Song-Tian \cite{ST} and Stoppa \cite{Stoppa1}); recently Chen \cite{Ch} introduced a new continuity path to approach the existence of CSCK via twisted CSCK metrics. When the first Chern class $c_1$ is positive, this new continuity path reduces to the classical Aubin path for the K\"ahler-Einstein equation. A main advantage of Chen's continuity path is that a twisted CSCK is a minimizer of twisted K-energy, which is always  strictly convex. In particular this would imply that the kernel of the linearized operator (fourth order) is always zero and hence openness holds. For more details and discussion, see Chen \cite{Ch}. Consider the twisted constant scalar curvature equation, 
\begin{equation}\label{CSCK}
t(R_\psi-\underline R)+(1-t)(\tr_{\omega_\psi}\omega-n)=0.
\end{equation}
By the result of the second author \cite{Zeng} and \cite{Hashimoto}, this equation is always solvable for $t$ near zero. Moreover openness holds for this path \cite{Ch} for $t>0$. If the equation has a smooth solution up to $t=1$, the the solution gives a desired CSCK. Hence we can assume that the maximal interval can be solved is $[0, T)$ for some $T\in (0, 1]$ and want to study what happens at the maximal singular time $T$. 
We have the following a priori estimate,

\begin{thm}\label{estimate}When  $t\rightarrow T$, then $\Delta \psi\rightarrow \infty$. In other words, if we assume that $\Delta \psi$ is bounded above, then there exists a constant $C$ such that
\[
\frac{1}{C}\omega\leq \omega_\psi\leq C\omega,\]
where $C$ is a uniform constants depending on $n, \omega, \sup \Delta\psi$. Moreover for $k\geq 3$,
\[
\|\psi\|_{C^k}\leq C_k,
\]
where $C_k$ is a uniform constant depending on $n, \omega, \sup \Delta\psi$ and $k$. 
\end{thm}

Hence one can solve the equation \eqref{CSCK} if one can show that $\Delta \psi$ is uniformly bounded for all $t$. A point to make is that we do not assume that $\omega_\psi$ is uniformly bounded away from zero, which can be obtained via the equation. Indeed by a rather straightforward maximum principle argument,  we obtain a definite lower bound of log-volume ratio. 
With $C^{1, \bar 1}$ assumption, this gives the desired strict positive. This should provide technical evidence to support Chen's conjecture. However, our result is merely \emph{a priori} estimate, there is subtlety to prove a regularity result to achieve such a positivity estimate for $C^{1, \bar 1}$ $K$-energy minimizers. We will address this difficulty in future work. \\

{\bf Acknowledgement:} Both author thank Prof. Xiuxiong Chen  for sharing his insight on this problem and for valuable suggestions and encouragements. The first author wants to thank Prof. Jingyi Chen and Micah Warren for valuable discussions of their work \cite{CW}, which motivates us to consider the local regularity for weak solution of constant scalar curvature equation in Section 3. The first author is support in part by an NSF grant, award no. 1611797. 

\section{Weak solutions of scalar curvature equation}
In this section we discuss various weak solutions of CSCK. First we introduce several norms on $\cH_\omega$
Define the weak $C^{1, 1}_w$ norm by
\[
\|\phi\|_{C^{1, 1}_w}=\|\phi\|_{L^\infty}+\|\p\bar \p \phi\|_{L^\infty}.
\]
We consider two extensions of the space of K\"ahler potentials,
\[
\cH^{1, 1}_{w}=\{\phi: \omega_\phi\geq 0, \|\phi\|_{C^{1, 1}_w}<\infty\}. \; \cH^{1, 1}_{w+}=\{\phi: \text{essinf}\;\omega_\phi>0, \|\phi\|_{C^{1, 1}_w}<\infty\}. 
\]
The difference is that we require for $\cH^{1, 1}_{w+}$, the positivity $\text{essinf}\; \omega_\phi>0$ holds as an $L^\infty$ matrix in local charts. In other words, a K\"ahler potential in $\cH^{1, 1}_{w+}$ defines an $L^\infty$ K\"ahler metric in $[\omega]$, while $\cH^{1, 1}_w$ only requires the current $ \omega_\phi$ is positive.

We first consider several notions of weak solutions of the Laplacian equation, $\Delta_\omega u=f,$ in particular we allow $\omega$ to be an $L^\infty$ metric on $M$. 
We will use these notions to define weak solutions of CSCK. 
The following weak solution is standard, and we recall its definition, 

\begin{defn}A $W^{1, 2}$ weak solution $u$ of $\Delta_\omega u=f$, for $f\in L^2$ if for any smooth (or $W^{1, 2}$) test function $v$, we have
\begin{equation}\label{E1}
\int_M \left(g^{i\bar j} \p_i u \p_{\bar j} v +fv\right) =0,
\end{equation}
where the integration is taken with respect to the volume form $dvol_g=\omega^n (n!)^{-1}$.
\end{defn}
This is  the well-known weak solutions of elliptic equations in divergence form, locally it can be written as \eqref{w2}. 
We need to also consider a notion of weak solution as follows, 
\begin{defn}\label{delta}Let $u, f\in L^1$. If for any smooth (or $W^{2,\infty}$) test function $v$, we have 
\begin{equation}\label{E2}
\int_M \left(g^{i\bar j} \p_i \p_{\bar j} v\right) u=\int_M f v
\end{equation}
then we call $u$ is a $\Delta$-weak solution of $\Delta_\omega u=f$. 
\end{defn}

In other words, $\Delta$-weak solution is defined in the sense of distribution. 
If $u, f$ has stronger regularity, for example $u\in L^\infty$, we can require test functions $v\in W^{2, p}$ for any $p\geq 1$ via an approximation by smooth functions. 
We emphasize that when $\omega$ (or $g$) is an $L^\infty$ metric, the definition of the weak solutions as above makes sense. 

Now we  consider two versions of weak solution of the constant scalar curvature metric, using the linear equations as above.
Let $V_\phi$ be the volume ratio given by $$V_\phi=\log \frac{\det(g_{i\bar j}+\phi_{i\bar j})}{\det(g_{i\bar j})}.$$
The scalar curvature equation can be written as
\begin{equation}\label{scalar}
R_\phi=-\Delta_\phi V_\phi+g^{i\bar j}_\phi R_{i\bar j}=f,
\end{equation}
where $R_{i\bar j}=-\p_i\p_{\bar j}\log\det(g_{i\bar j})$ is the Ricci curvature of $\omega$.

\begin{defn}First assume that $\phi\in \cH^{1, 1}_{w+}\cap W^{3, 2}$, then $V_\phi\in W^{1, 2}$. Then $\phi$ is called a standard weak solution if $V_\phi$ is a $W^{1, 2}$-solution of \eqref{scalar} with respect to $\omega_\phi$. That is, for any smooth testing function $v$, we have
\begin{equation}
\int_M \left(g^{i\bar j}_\phi \p_i V_\phi \p_{\bar j} v+g^{i\bar j}_\phi R_{i\bar j} v-fv\right)dvol_\phi=0. 
\end{equation}
\end{defn}

Then the following regularity result is rather standard, using well-known regularity theory in the field.

\begin{prop}\label{psmooth}Suppose $\phi\in \cH^{1, 1}_{w+}\cap W^{3, 2}$ is a standard weak solution of $R_\phi=\underline{R}$, then $\phi$ is smooth and it defines a smooth K\"ahler metric with constant scalar curvature. 
\end{prop}

\begin{proof}Since $\phi$ defines an $L^\infty$ K\"ahler metric, by the De Giorgi-Nash-Moser theory, a $W^{1, 2}$ solution of the form \eqref{E1} is locally H\"older continuous. Hence this implies that $V_\phi\in C^{\alpha}$, which in turn implies that $\phi\in C^{2, \alpha}$, by the theory of complex Monge-Ampere equation \cite{wangyu}. Then rather standard Schauder theory and boot-strapping argument imply that $\phi$ is smooth. 
\end{proof}

Now we drop the assumption that $\phi\in W^{3, 2}$. We define the weak solution, using the $\Delta$-weak solution of the linear equation above. 
\begin{defn}
We define a weak solutions of $R_\phi=f$ for $\phi\in \cH^{1, 1}_{w+}$, if for any function $\psi\in C^{1, 1}$, we have
\[
-\int_M V_\phi \Delta_\phi \psi \omega^n_\phi=\int_M \psi (f- g^{i\bar j}_\phi R_{i\bar j})\omega^n_\phi.
\]
In other words, $V_\phi$ is a $\Delta$-weak solution, with respect to $L^\infty$ metric $\omega_\phi$, 
\[
\Delta_\phi V_\phi=g^{i\bar j}_\phi R_{i\bar j}-f.
\]
\end{defn}

Clearly, the definition can be made local by requiring that  $\psi$ has compact support. 
We will simply refer a weak solution of constant scalar curvature equation to the above definition. 
We would like to ask the following,
\begin{prob}\label{p1}
Suppose $\phi\in \cH^{1, 1}_{w+}$ with bounded $C^{1, 1}_w$ norm and $\phi$ is a weak solution of $R_\phi=\underline{R}$. Is $\phi$ smooth?
\end{prob}

To answer this question, one needs to show that $V_\phi$, as a $\Delta$-weak solution, is indeed a $W^{1, 2}$-solution by Proposition \ref{regularity}. 
This question is directly related to Chen's conjecture \cite{chen08}. To see the relation, we have the following simple observation, 

\begin{prop}
If $u\in \cH^{1, 1}_{w+}$ is a weak solution of the constant scalar curvature equation, then it is a critical point of $K$-energy (among smooth variations). In particular, a K-energy minimizer (or critical point) in 
$\cH^{1, 1}_{w+}$ is a weak solution.
\end{prop}

\begin{proof}
Recall Chen's formula for $\cK$-energy, 
\begin{equation}
\cK(u)=\frac{\underline R}{n+1}\cE(u)-\cE_{Ric} (u)+\int_M \log\frac{\omega_u^n}{\omega^n} \omega_u^n,
\end{equation}
where $\cE(u)$ and $\cE_{Ric}(u)$ are two well-known energy functionals, given by
\[
\cE(u)=\frac{1}{V}\sum_{j=0}^n\int_M u \omega_u^j\wedge \omega^{n-j}\;\;; \cE_{\alpha} (u)=\frac{1}{V}\int_M u\alpha\wedge\omega_u^j\wedge \omega^{n-j-1}
\]
One can compute the first variation of $\cK$ by the following,
\begin{equation}\label{Eweak}
\delta \cK=\underline{R}\int_M \psi  \omega^n_u-n\int_M \psi Ric\wedge \omega^{n-1}_u+n\int_M \sqrt{-1}\p\bar\p \psi \wedge \omega_u^{n-1} \log\frac{\omega_u^n}{\omega^n}
\end{equation}
If $\omega_u$ defines a smooth K\"ahler metric, then we can continue to compute
\[
\delta \cK=-\int_M \psi (R_u-\underline R)\omega_u^n
\]
Hence it leads to the Euler-Lagrangian equation,
\[
R_u=\underline{R}.  
\]
When $\omega_u$ is only an $L^\infty$ K\"ahler metric, it leads to a weak solution in distribution sense (called $\Delta$-weak solution) of $R_u=\underline R$, for any $\psi\in C^\infty$, 
\begin{equation}
\int_M \Delta_u \psi  \log\frac{\omega_u^n}{\omega^n} \omega_u^n=\int_M \psi \left(\tr_{\omega_u} Ric-\underline{R}\right)\omega_u^n 
\end{equation}
\end{proof}

We will answer Problem \ref{p1} affirmatively in Section 4.  By Proposition \ref{psmooth}, the key is to show that the volume ratio $V_\phi$ is in $W^{1, 2}$. In next section we will first prove that given a smallness condition, a purely local  $W^{1, 2}$  regularity for $V_\phi$.

\section{Local regularity with smallness and removable singularities}

In this section we  consider the regularity problem purely in local, that we have a metric defined in an Euclidean ball $B(1)$, by $\left(u_{i\bar j}\right)>0.$
We can define a weak solution of 
\[
R_u=f
\] for $u_{i\bar j}\in L^\infty$,  given any smooth functions $\psi$ with compact support in $B(1)$,

\begin{equation}\label{weak1}
-\int_B \log (\det (u_{i\bar j})) \det(u_{i\bar j}) u^{i\bar j}\p^2_{i\bar j}(\psi)= \int_B \psi f \det(u_{i\bar j})
\end{equation}

We introduce some notations.  We will use both real and complex coordinate. For example, we use $u_i, u_{\bar j}$ to denote derivatives in complex variable $\p_{z_i} u, \p_{\bar z_j}u$,  $(1\leq i, j\leq n)$, and $u_\alpha$  to denote $\p_{x_\alpha} u$, $1\leq \alpha\leq 2n$, to denote derivatives in real variables $x=(x_\alpha)$. 
For a function $f$, we denote the difference quotient 
\[\delta^h_{\alpha} f= \frac{f(x+he_\alpha)-f(x)}{h}\]

Then we have the following,
\begin{thm}\label{regularity1}
There exists a constant $c_n=c(n)$, which can be computed explicitly such that, if $(1-c_n)\delta_{i\bar j}\leq u_{i\bar j}\leq (1+c_n)\delta_{i\bar j}$, then $u\in W^{3, 2}_{loc}$ and in particular, $u$ is a smooth solution of $R_u=f$ if $f$ is smooth.
\end{thm}

\begin{proof}The key is to show that $u\in W^{3, 2}_{loc}$. By an approximation, we can take testing functions in $W^{2, \infty}_0$. Let $\eta$ be a cut-off function with compact support in $B(1)$. We take, for some small $h>0$, which will be specified later (it depends on the support of $\eta$),   
\[
\psi=\delta^{-h}_{\alpha} (\eta^4\delta^h_{\alpha} u),\; v=\delta^h_{\alpha} u
\]
We compute, by 
\[
\int_B  \log (\det (u_{i\bar j})) \det(u_{i\bar j}) u^{i\bar j}\p^2_{i\bar j}(\delta^{-h}_\alpha (\eta^4\delta^h_{\alpha} u))=\int_B \delta^h_{\alpha} (U_{i\bar j}) \p^2_{i\bar j} (\eta^4 \delta^h_{\alpha} u),
\]
where $U_{i\bar j}=\log (\det (u_{i\bar j})) \det(u_{i\bar j}) u^{i\bar j}$, for a given matrix $(u_{i\bar j})$.

Now consider the matrix $\p\bar\p w=w_{i\bar j}(x, t)=t u_{i\bar j}(x+he_\alpha)+(1-t)u_{i\bar j}(x)$.
Denote \[U_{i\bar j}(t)=\log (\det (w_{i\bar j}(t))) \det(w_{i\bar j}(t)) w^{i\bar j}(t)\] 
Then we compute
\[
\begin{split}
\delta_\alpha^h U_{i\bar j}=&\frac{U_{i\bar j}(1)-U_{i\bar j}(0)}{h}\\
=&h^{-1}\int_0^1 U^{'}(t)dt\\
=&\int_0^1 \frac{\p U_{i\bar j}}{\p w_{k\bar l}} dt \frac{u_{k\bar l}(x+he_\alpha)-u_{k\bar l}(x)}{h}
\end{split}
\]

Then we compute
\[
\frac{\p U_{i\bar j}}{\p w_{k\bar l}}=w^{k\bar l} \det(w_{i\bar j}) w^{i\bar j}+\log\det(w_{i\bar j})\det(w_{i\bar j}) \left(w^{k\bar l}w^{i\bar j}- w^{i\bar l}w^{k\bar j}\right)
\]

Denote \begin{equation}
\label{A}a^{i\bar j}_{k\bar l}(\p\bar\p w)=w^{k\bar l} \det(w_{i\bar j}) w^{i\bar j}+\log\det(w_{i\bar j})\det(w_{i\bar j}) \left(w^{k\bar l}w^{i\bar j}- w^{i\bar l}w^{k\bar j}\right)\end{equation}

It follows that, by mean value theorem, for some $t_0\in [0, 1]$ 

\[
\int_0^1 \frac{\p U_{i\bar j}}{\p w_{k\bar l}} dt=\int_0^1a^{i\bar j}_{k\bar l} (\p\bar\p w(t))dt=a^{i\bar j}_{k\bar l} (\p\bar\p w(t_0)). 
\]

We compute, with $v=\delta^h_\alpha$

\[
\int_B \delta^h_{\alpha} (U_{i\bar j}) \p^2_{i\bar j} (\eta^4 \delta^h_{\alpha} u)=\int_B a^{i\bar j}_{k\bar l}v_{k\bar l}\left( v_{i\bar j}  \eta^4+v\p^2_{i\bar j}(\eta^4)+v_i\p_{\bar j}(\eta^4)+v_{\bar j}\p_i(\eta^4)\right)
\]

First we estimate the following, for any $\epsilon\in (0, 1)$, there exists a constant $C=C(n, \eta, \epsilon)$, 

\begin{equation}
\int_B a^{i\bar j}_{k\bar l}v_{k\bar l}\left( v\p^2_{i\bar j}(\eta^4)+v_i\p_{\bar j}(\eta^4)+v_{\bar j}\p_i(\eta^4)\right)\leq \epsilon \int_B |v_{k\bar l}|^2\eta^4+C\int_B \left(|\nabla v|^2+v^2\right)
\end{equation}

This is by straightforward computations. For the first term,
\[
\begin{split}
\left|\int_B a^{i\bar j}_{k\bar l}v_{k\bar l}v\p^2_{i\bar j}(\eta^4)\right|\leq &C\int_B |v_{k\bar l}| v\eta^2 (|\nabla \eta|^2+\eta|\nabla^2\eta|)\\
& \leq \epsilon \int_B |v_{k\bar l}|^2 \eta^4+C(n, \eta, \epsilon) \int_B v^2
\end{split}
\]
The other two terms can be handled similarly. 
Next we claim that there exist $c_1=c_1(n), C_1=C_1(n, \eta)$, when $c_n$ is small, then we have, 
\begin{equation}\label{mainestimate}
\int_B a^{i\bar j}_{k\bar l}v_{k\bar l} v_{i\bar j}\eta^4\geq c_1\int_B |\nabla^2 v|^2\eta^4-C\int_B \left(|\nabla v|^2+v^2\right),
\end{equation}

This is the place where we need to use very basic Caldron-Zygmund theory. Suppose we have established the above claim, then it follows that, on the set $\Omega=\{\eta=1\}$
\[
\int_\Omega |\nabla^2 v|^2\leq C.
\]
This implies $u\in W^{3, 2}_{loc}$. Once $u\in W^{3, 2}_{loc}$, it is smooth by the well-known theory in the subject. We roughly go through the argument: $u\in W^{3, 2}_{loc}$, then it implies that $V=\log\det(u_{i\bar j})\in W^{1, 2}_{loc}$ and it is a weak solution of the equation
\[
\Delta_u V=-f
\]
Since $u_{i\bar j}$ is uniformly bounded and positive, it implies $V\in C^\alpha$ by standard elliptic theory. By the complex Monge-Ampere equation $\log\det(u_{i\bar j})=V$, it then follows that $u\in C^{2, \alpha}$ \cite{wangyu}. Hence $u$ is then smooth by standard bootstrapping argument. 
 \end{proof}

We prove the following proposition.  

\begin{prop}\label{CZ}
Given $c_n$ sufficiently small, then the estimate \eqref{mainestimate} holds. 
\end{prop}
\begin{proof}
By assumption, we know that $(1-c_n)\delta_{i \bar j}\leq w_{i\bar j}(t)\leq (1+c_n)\delta_{i \bar j}$, for any $t\in [0, 1]$. Hence we compute
\[
\begin{split}
&(1-c_n)^n\leq \det(w_{k\bar l})\leq (1+c_n)^n,\\
&|\log(\det(w_{k\bar l}))|\leq nc_n,\\
&w^{k\bar l} v_{k\bar l}=\Delta v+\left(w^{k\bar l}-\delta^{k\bar l}\right) v_{k\bar l}
\end{split}
\]
We can estimate, using $(a+b)^2\geq \frac{1}{2}a^2-3b^2$,
\[
\begin{split}
\left|w^{k\bar l} v_{k\bar l}\right|^2\geq \frac{1}{2}(\Delta v)^2-3\left|w^{k\bar l}-\delta^{k\bar l}\right|^2 |v_{k\bar l}|^2\geq \frac{1}{2}(\Delta v)^2-\frac{3c_n^2}{(1-c_n)^4}|\nabla^2 v|^2
\end{split}
\]
Then we get a pointwise estimate, using \eqref{A}, 
\[
a^{i\bar j}_{k\bar l}v_{k\bar l}v_{i\bar j}\geq a_n (\Delta v)^2-b_n|\nabla^2 v|^2,
\]
where we compute,
\[
\begin{split}
&a_n=\frac{1}{2}(1-nc_n)(1-c_n)^n\\
&b_n=3c_n^2(1-nc_n)(1-c_n)^{n-4}+nc_n(1+c_n)^n
\end{split}
\]
Hence we get 
\begin{equation}\label{mainestimate2}
\int_B a^{i\bar j}_{k\bar l}v_{k\bar l} v_{i\bar j}\eta^4\geq \int_B \left(a_n (\Delta v)^2-b_n|\nabla^2 v|^2\right)\eta^4
\end{equation}
We need a special case of an elementary Caldron-Zygmund equality (inequality). We have for $w\in W^{2, 2}_0$, 
\begin{equation}\label{CZ1}
\int_B (\Delta w)^2=\int_B |\nabla^2 w|^2.
\end{equation}
We need a straightforward extension,  for any $\epsilon\in (0, 1)$, there exists $C=C(n, \eta, \epsilon)$, such that
\begin{equation}\label{CZ2}
\int_B (\Delta v)^2\eta^4\geq (1-\epsilon)\int_B |\nabla^2 v|^2\eta^4-C(n, \eta, \epsilon)\int_B \left(|\nabla v|^2+\int_B v^2\right). 
\end{equation}
We compute,
\[
\begin{split}
\int_B (\Delta v)^2\eta^4=&\int_B |\Delta (v\eta^2)|^2-\int_B \left(\Delta(\eta^2) v+\nabla(\eta^2)\nabla v \right)^2\\
&-2\int_B \Delta v \eta^2 \left(\Delta(\eta^2) v+\nabla(\eta^2)\nabla v \right)\\
\geq & (1-\epsilon_1)\int_B |\Delta (v\eta^2)|^2-C \int_B \left(\Delta(\eta^2) v+\nabla(\eta^2)\nabla v \right)^2\\
=&(1-\epsilon_1)\int_B|\nabla^2 (v\eta^2)|^2-C \int_B \left(\Delta(\eta^2) v+\nabla(\eta^2)\nabla v \right)^2\\
\geq&(1-2\epsilon_1)\int_B |\nabla^2 v|\eta^4-C(n, \eta, \epsilon_1)\int_B \left(|\nabla v|^2+\int_B v^2\right). 
\end{split}
\]
It then follows, from \eqref{mainestimate2} and \eqref{CZ2}, that
\begin{equation}
\int_B a^{i\bar j}_{k\bar l}v_{k\bar l} v_{i\bar j}\eta^4\geq 
c_1\int_B |\nabla^2 v|^2 \eta^4-C\int_B (|\nabla v|^2+v^2),\end{equation}
with $c_1=a_n(1-\epsilon)-b_n$. Take $\epsilon=1/2$ for example,  then $c_1>0$ when $c_n$ is small. This completes the proof.  
\end{proof}

The following then follows directly from our local regularity result. 
\begin{cor}\label{regularity}Given $(M, g)$ and suppose $\psi\in \cH^{1, 1}_{w+}$ is a weak solution . Then there exists a positive dimensional constant $c_0=c_0(n)$ such that if  $\|\p\bar \p \psi\|_g\leq c_0$, then $\psi$ is smooth. 
\end{cor}
\begin{proof}It is standard to localize the problem, by using the K\"ahler condition. We roughly sketch a proof. Locally we can have a coordinate patch, that the metric $g$ is given by a local potential, $g_{i\bar j}=\p_i\p_{\bar j} h$, for some $h$ smooth and $\sqrt{-1}\p\bar\p h>0$. Then locally the metric $g_\psi$ is given by $\sqrt{-1}\p\bar \p (h+\psi)$.We can pick up finitely many open sets $U_i$ to cover the manifold. By a scaling argument, we can assume that $U_i$ are holomorphic to Euclidean balls $B(1)$.   We can also assume that $h_{i\bar j}$ is sufficiently closed to $\delta_{i\bar j}$, for some $\epsilon$ sufficiently small,  $(1-\epsilon)\delta_{i\bar j}<h_{i\bar j}<(1+\epsilon)\delta_{i\bar j}$. Note that we can make $\epsilon$ a universal constant, even though the open covering $U_i$ clearly depends on the metric $g$. 
Then we apply our argument locally to $u=h+\psi$, then it follows that there exists a constant $c_1$, such that is $-c_1\delta_{i\bar j}\leq \psi_{i\bar j}<c_1\delta_{i\bar j}$, then $(1-c_n)\delta_{i\bar j}\leq h_{i\bar j}+\psi_{i\bar j}\leq (1+c_n)\delta_{i\bar j}$. Then it follows that $\psi$ is smooth on each local patch, hence smooth on the manifold. Note that $c_1$ depends only on $c_n$ and $\epsilon$, hence it is a dimensional constant. Then we need to note that $\|\p\bar\p \psi \|_g$ does not depend on the choice of local coordinates, and hence we can choose $c_0=c_1$. \end{proof}

\begin{rmk}
Recently Chen-Warren \cite{CW} consider a fourth order equation: hamiltonian stationary equation and proved a regularity result with small $L^\infty$ on Hessian. Our regularity result mimics theirs, even though there is technical difference since the nonlinearity is quite different. For example, we need a very basic Caldron-Zygmund inequality to prove Proposition \ref{CZ}. \end{rmk}

In \cite{CW} Chen-Warren proved a version of  removable singularity for the equation by allowing the equation is defined away a compact capacity zero set, using a result of Serrin \cite{Serrin}. We refer to \cite{CW, Serrin} for the notion of capacity and their removable singularity results. Similar results also hold in our case. In particular we consider only the equation of the form, in weak sense of \eqref{weak1},
\[
R_u=\underline{R}
\]
We can rewrite the equation as, in the weak sense of \eqref{weak1}, 
\[
-\Delta_u (\log(\det(u_{i\bar j}))+\underline{R} u)=0
\]
\begin{thm}Let $Q\subset B=B(1)$ be a compact set with capacity zero. Then there exists $c_n$ such that $u\in C^{1, 1}(B\backslash Q)$ is a weak solution of \eqref{weak1} with the following holds (in the sense of $L^\infty$), 
\[
(1-c_n)\delta_{i \bar j}\leq u_{i\bar j}\leq (1+c_n)\delta_{i \bar j},
\]
then $u$ is smooth solution in $B$.
\end{thm}
\begin{proof}We can apply Theorem \ref{regularity1} locally for any point $p\in B\backslash Q$ locally on a ball $B_p(r)$ with radius $r$,   $B_p(r)\subset B\backslash Q$, since $Q$ is compact. 
Indeed in the ball $B_p(r)$, we can consider $u_r(x)=r^{-2} u(p+rx)$, for $x\in B(1)$. An easy computation shows that $\p\bar\p u_r (x)=\p\bar\p u (p+r x)$. Hence $u_r(x)$ is a weak solution of \eqref{weak1} in $B(1)$ and satisfies the assumption in Theorem \ref{regularity1}. It follows that $u_r$ is smooth in $B(1)$, hence $u$ is smooth in $B\backslash Q$. To extend $u$ across $Q$, we need to use a theorem of Serrin \cite[Theorem 1.2]{Serrin}. If $u$ satisfies \eqref{weak1}, then for any $\eta$ supported in $B$ away from $Q$, we have
\[
\int_B u^{i\bar j} \p_i\p_{\bar j} (\eta)  \theta \det(u_{i\bar j}) dx=0,
\]
where $\theta=\log(\det(u_{i\bar j}))+\underline{R} u$. Since $u$ is smooth in $B\backslash Q$ and $\eta$ is supported in $B$ away from $Q$, then we compute
\[
\int_B \eta_{i} u^{i\bar j}\det(u_{i\bar j}) \theta_{\bar j} dx=0.
\]
In other words, $\theta$ is a weak solution of $\p_i (u^{i\bar j}\det(u_{i\bar j}) \p_{\bar j}\theta)=0$.  Serrin's result then implies that $\theta$ can be extended across $Q$, as a weak solution $\p_i (u^{i\bar j}\det(u_{i\bar j}) \p_{\bar j}\theta)=0$ on $B$. A standard elliptic regularity then implies that $\theta\in C^{\alpha}(B)$, by de Georgi-Nash-Moser type of estimates. Then it follows that $u$ is smooth on $B$, arguing as in Theorem \ref{regularity1}. 
\end{proof}

\section{$W^{2, 2}$ estimate and CSCK on compact K\"ahler manifolds}

In this section we derive $W^{2, 2}$ estimate of the linear equation of the form $\Delta_\omega u=f$ with $L^\infty$ coefficients, and give a proof of Theorem \ref{main0}. Theorem \ref{main} is a direct consequence of Theorem \ref{main0} 

By the discussion in Section 2, the key for the smooth regularity  is to show the volume ratio is in $W^{1, 2}$; in other words, whether a $\Delta$-weak solution is also $W^{1, 2}$ weak solution (for  linear equation). By a simple observation, we see that a $W^{1, 2}$ weak solution is always a $\Delta$-weak solution.

Suppose $\phi\in \cH_{\omega_+}$. Then $\omega_\phi$ defines an $L^\infty$ K\"ahler metric and we also assume the following bound through this section, for two positive constants $\epsilon<\Lambda$
\[
\epsilon \omega\leq \omega_\phi\leq \Lambda \omega. 
\]

\begin{prop}\label{prop2.3}
Suppose $\phi \in \cH_{\omega+}^{1,1}$. For any $u \in W^{1,2}(M)$ and any $\psi \in C^\infty(M)$, we have
\begin{equation}
\int_M d u \wedge d^c \psi \wedge \omega_\phi^{n-1} =  - \int_M u dd^c \psi \wedge \omega_\phi^{n-1}, 
\end{equation}
where $d = \partial + \bar \partial $ and $d^c = \sqrt{-1}(\bar \partial - \partial )$. Consequently, if $u$ is a $W^{1,2}$ weak solution of equation $\Delta_{\phi} u = f$, then $u$ is also a $\Delta$-weak solution.
\end{prop}

An interesting question is the following,
\begin{prob}\label{weakharmonic}Let $u\in L^1$ be a $\Delta$-weak solution of $\Delta_\omega u=0$. Is $u$ a constant?
\end{prob}
If this uniqueness statement were true, then we can solve a $W^{1, 2}$ weak solution of the form \eqref{w2}, and the uniqueness above guarantees that a $\Delta$-weak solution coincide with a $W^{1, 2}$ solution. This would be sufficient for the proof. 
In Euclidean space, Weyl's lemma asserts that a $L^1$ $\Delta$-weak harmonic function is smooth (one requires testing functions have compact support). 
Similarly Problem \ref{weakharmonic} is true when $\omega$ is smooth since the image of $\Delta_\omega$ contains all smooth functions (continuous functions) with average zero. One can  view Problem \ref{weakharmonic} as a global (and easier) version of Weyl's lemma if $\omega$ is smooth. 
One can also answer Problem \ref{weakharmonic} affirmatively for $\omega$ continuous, or even when $\omega$ is an $L^\infty$ metric with various smallness assumptions, such as $\omega$ is in VMO. The essential role is the notion of strong solution in non-divergence form with non-continuous coefficients ($L^\infty$ coefficients). Note that $W^{2, p}$ estimates 
holds for continuous elliptic coefficients, or more generally for $L^\infty$ elliptic coefficients with certain smallness assumption, such as $VMO$ \cite{CFL}. But it
generally fails with only elliptic and $L^\infty$ coefficients(\cite{PT}). Hence Problem \ref{weakharmonic} is  subtle  if we require only $\omega$ to be an $L^\infty$ K\"ahler metric. 

Our observation is, on a compact K\"ahler manifold, for an $L^\infty$ metric defined by $\phi \in \cH^{1,1}_{\omega+}$, we have the following $W^{2,2}$ regularity estimate for a $W^{1,2}$ weak solution for equation $\Delta_{\phi} u = f$. Our $W^{2,2}$ estimate depends only on the uniform ellipticity of $\omega_\phi$, with no smallness assumption on its coefficients. A classical result is that $W^{2, 2}$ estimate holds for uniform elliptic equation with bounded coefficients when (real) dimension is $2$. Our result on K\"ahler manifolds should be viewed as a generalization of such a result since in real dimension 2, every $L^\infty$ metric is indeed K\"ahler (positive definite elliptic coefficients can be interpreted as a metric). 
\begin{thm}\label{prop2.6}
Suppose $\phi \in \cH^{1,1}_{\omega+}$. Suppose $v$ is the $W^{1,2}$ weak solution of equation $\Delta_\phi v = f$ for some $f\in L^2(M)$ with $\int_M f \omega_\phi^n =0$. Then $v \in W^{2,2}(M)$ and 
\begin{equation}
\Delta_{\phi} v = f, 
\end{equation}
holds almost everywhere on $M$. 
\end{thm}

As a direct consequence, we have the following uniqueness result of $\Delta$-weak solution of $\Delta_\phi u = f$ for $\phi \in \cH^{1,1}_{\omega+}$. 
\begin{cor}\label{weakuniqueness}
Suppose $\phi \in \cH^{1,1}_{\omega+}$. If $u \in L^2(M)$ and $u$ is a $\Delta$-weak solution of $\Delta_\phi u = 0$, then $u$ must be a constant. 
\end{cor}

\begin{proof}
Given any $f \in L^{2}(M)$ with $\int_M f \omega_\phi^n = 0$, one could find the $W^{1,2}$ weak solution of $\Delta_\phi u = f$, denoted by $v \in W^{1,2}(M)$. By Theorem $\ref{prop2.6}$, we know that $v \in W^{2,2}(M)$ and $\Delta_\phi v = f$ holds almost everywhere. One can find a smooth approximation of $v$, say $v_\epsilon$, converging to $v$ in $W^{2,2}$ sense. Thus we have
\begin{equation}
\int_M u (\Delta_\phi v_\epsilon) \omega_\phi^n = 0. 
\end{equation}
Since $u \in L^2(M)$, by letting $\epsilon \rightarrow 0$, we have
\begin{equation}
\int_M u f \omega_\phi^n = 0, 
\end{equation}
for any $f \in L^2(M)$ with $\int f \omega_\phi^n = 0$. Such $u$ must be a constant.
\end{proof}

We first make some preparations by introducing the following lemmas and then we will prove Theorem $\ref{prop2.6}$.

\begin{lem}\label{lem2.5}
Given any $\phi \in \cH^{1,1}_{\omega+}$, there exists a smooth family of smooth K\"ahler potentials $\phi_t (0< t \ll 1)$ such that as $t \rightarrow 0^+$,
\begin{enumerate}
\item[i)] $\phi_{t} \rightarrow \phi$ in $W^{2,p}$ space for any $1< p < \infty$; 
\item[ii)]  $\big(\text{essinf } \omega_\phi - o(1) \big) \omega_0 \leq \omega_{\phi_t} \leq \big(\text{esssup} \omega_\phi + o(1)\big) \omega_0$ on $M$. 
\end{enumerate} 

\end{lem}

\begin{proof}
Denote $H \in C^\infty\big(M \times M \times (0,\infty) \big)$ to be the heat kernel of a smooth Riemannian metric, for instance the background metric $\omega_0$. Define 
\begin{equation}
\phi_t(x) := \int_M H(x, y; t) \phi(y) \omega_0^n(y). 
\end{equation} 
We could verify directly that $\phi_t$ defined as above satisfies i) and ii). 

\end{proof}

\begin{lem}\label{lem2.6}
For any $L^\infty$ Riemannian metric $g$, we define
\begin{equation}
\lambda_1(g) := \inf \{ \frac{\int_M |\nabla u|_g^2 dvol_g}{\int_M u^2 dvol_g} \big| u \in W^{1,2}(M), \int u dvol_g = 0. \} 
\end{equation}
which is a positive constant depending on $g$. For any $K >1$, there exists a constant $\epsilon_{K, n} > 0$ such that for any $L^\infty$ Riemannian metric $h$ with $K^{-1} g \leq h \leq K g$, we have 
\begin{equation}
\lambda_1(h) \geq \epsilon_{K,n} \lambda_1(g).
\end{equation}

\end{lem}

\begin{proof}
Given any $u \in W^{1,2}(M)$ with $\int_M u dvol_h = 0$, we consider
\begin{equation}
\frac{\int_M |\nabla u|_{h}^2 dvol_h}{\int_M u^2 dvol_h } \geq \epsilon_{K, n} \frac{\int_M |\nabla u|_{g}^2 dvol_g}{\int_M u^2 dvol_h}
\end{equation}
Denote $\tilde u = u - \frac{1}{V(g)} \int_M u dvol_g$ and then we have
\begin{equation}
 u = \tilde u - \frac{1}{V(h)}\int_M \tilde u dvol_h. 
\end{equation}
Thus
\begin{equation}
\begin{split}
\int_M u^2 dvol_h = & \int_M (\tilde u - \frac{1}{V(h)}\int_M \tilde u dvol_h)^2 dvol_h \\
= & \int_M \tilde u^2 dvol_h - \frac{1}{V(h)}(\int_M \tilde u dvol_h)^2\\
\leq & C_{K, n} \int_M \tilde u^2 dvol_g. 
\end{split}
\end{equation}
Thus putting these estimates together, we have
\begin{equation}
\frac{\int_M |\nabla u|_{h}^2 dvol_h}{\int_M u^2 dvol_h } \geq \epsilon_{K, n} \frac{\int_M |\nabla \tilde u|_{g}^2 dvol_g}{\int_M \tilde u^2 dvol_g} \geq \epsilon_{K, n} \lambda_1(g).
\end{equation}

\end{proof}

\begin{proof}
Proof of Theorem $\ref{prop2.6}$. Denote $\lambda : = \text{essinf } \omega_\phi$ and $\Lambda : = \text{esssup } \omega_\phi$. 

Suppose $\phi_t$ for $0<t \ll1$ is the smooth approximation of $\phi \in \cH^{1,1}_{\omega+}$ obtained in Lemma $\ref{lem2.5}$ and $f_t (0< t \ll 1)$ is a smooth approximation of $f \in L^2(M)$ under the $L^2$ norm. Without loss of generality, we can assume $\int_M f_t \omega_{\phi_t}^n = 0$ for all $0 < t \ll 1$. Then for any $0 < t \ll 1$, there exists $v_t \in C^\infty(M)$ with $\int_M v_t \omega_{\phi_t}^n = 0 $ such that $\Delta_{\phi_t} v_t = f_t$ holds smoothly on $M$. 

We have
\begin{equation}
\int_M |\nabla v_t|_{\omega_{\phi_t}}^2 \omega_{\phi_t}^n = -\int_M f_t v_t \omega_{\phi_t}^n \leq C \|f_t\|_{L^2(M)} \big(\int_M v_t^2 \omega_{\phi_t}^n\big)^{\frac{1}{2}}, 
\end{equation}
for some constant $C$ depending on $\Lambda$. Since for $t \ll 1$, $g_{\phi_t}$'s are uniformly bounded, by Lemma $\ref{lem2.6}$, there exists a constant $C$ depending on $\lambda, \Lambda$ and dimension $n$ such that
\begin{equation}
 \int_M v_t^2 \omega_{\phi_t}^n \leq C \int_M |\nabla v_t|_{\omega_{\phi_t}}^2 \omega_{\phi_t}^n. 
\end{equation}
Thus we have for $t \ll 1$,
\begin{equation}
\|v_t\|_{W^{1,2}(M)} \leq C \|f\|_{L^2(M)}.
\end{equation}
for some constant $C> 0$ depending on $\lambda, \Lambda$ and dimension $n$.

Compute for $t$ sufficiently small, 
\begin{equation}
\int_{M} g_{\phi_t}^{i\bar l} g_{\phi_t}^{k \bar j} v_{t, i\bar j} v_{t, k\bar l} \omega_{\phi_t}^n = \int_M ( \Delta_{\phi_t} v_t)^2 \omega_{\phi_t}^n = \int_M f_t^2 \omega_{\phi_t}^n \leq C \|f\|_{L^2(M)}^2,  
\end{equation}
and consequently
\begin{equation}
\limsup_{t \rightarrow 0^+ } \big\| |\partial \bar \partial v_t |_{\omega_0}\big\|_{L^2(M)} \leq C \|f\|_{L^2(M)}, 
\end{equation}
for some constant $C > 0$ depending on $\lambda$, $\Lambda$ and the dimension $n$ only. Also we have for $t$ sufficiently small
\begin{equation}
\begin{split}
\int_{M} g_0^{i\bar l} g_0^{k \bar j} v_{t, ik} v_{t, \bar l \bar j} \omega_0^n = & \int_M (\Delta_{\omega_0} v_t)^2 \omega_0^n - \int_M (\text{Ric}_{\omega_0})_{i\bar j} g_0^{i \bar p} g_0^{q \bar j} v_{t, \bar p} v_{t, q} \omega_0^n, \\
\leq  &C \big \||\partial \bar \partial v_t |_{\omega_0}\big \|_{L^2(M)}^2 + C \|v_t\|_{W^{1,2}(M)}
\end{split}
 \end{equation}
where subscript ``," means covariant derivative with respect to background metric $\omega_0$. Thus we have
\begin{equation}
\limsup_{t \rightarrow 0^+} \|v_t\|_{W^{2,2}(M)} \leq C \|f\|_{L^2(M)},
\end{equation}
for some constant $C > 0$ depending on $\lambda$, $\Lambda$, $n$ and background metric $\omega_0$ only.

Therefore, as $t \rightarrow 0$, $v_t$'s admits a convergent subsequence converging in $W^{1,2}$ norm to some limit function $v_0 \in W^{2,2}(M)$. For any fixed text function $\psi\in C^\infty(M)$, we have
\begin{equation}
\frac{1}{2}\int_M  { d} \psi \wedge d^c v_t \wedge \omega_{\phi_t}^{n-1} =  - \int_M f_t \psi \omega_{\phi_t}^n. 
\end{equation}
Letting $t \rightarrow 0$, we get 
\begin{equation}
\frac{1}{2}\int_M   d \psi \wedge d^c v_0 \wedge \omega_\phi^{n-1} = - \int_M f \psi \omega_{\phi}^n, 
\end{equation}
for any $\psi \in C^\infty(M)$. This tells us the limit $v_0$ is actually the unique $W^{1,2}$ weak solution to $\Delta_\phi u = f$, which equals to $v$. By our previous discussion, $v = v_0$ now has $W^{2,2}$ regularity, we have that for any $\psi \in C^{\infty}(M)$
\begin{equation}\label{eqn2.17}
\frac{1}{2}\int_M \psi dd^c v \wedge \omega_{\phi}^{n-1} =  -\frac{1}{2} \int_M d \psi \wedge d^c v \wedge \omega_\phi^{n-1} = \int_M f \psi \omega_{\phi}^n. 
\end{equation}
This is essentially because $\omega_{\phi}^{n-1}$ is a closed positive $(n-1, n-1)$ current. More precisely, similar to the argument in Proposition $\ref{prop2.3}$, we suppose $v_\epsilon$ is a smooth approximation of $v$ in $W^{2,2}$ space, and consider smooth $1$-form $\alpha = \psi d^c v_\epsilon$ with $d \alpha = d\psi \wedge d^c v_\epsilon + \psi dd^c v_\epsilon$. Then we have that
\begin{equation}
\int_M  ( d\psi \wedge d^c v_\epsilon + \psi dd^c v_\epsilon )\wedge \omega_{\phi}^{n-1} = \int_M  d\alpha \wedge \omega_{\phi}^{n-1}= 0.
\end{equation}
Let $\epsilon \rightarrow 0$, we have $(\ref{eqn2.17})$. 
Thus we have $v \in W^{2,2}(M)$ with $\Delta_\phi v = f$ almost everywhere.
\end{proof}

Now we can prove Theorem \ref{main0} and Theorem \ref{main}.

\begin{proof}[Proof of Theorem \ref{main0}]First an $L^\infty$ K\"ahler metric which minimizes $\cK$ energy is a $\Delta$-weak solution of CSCK, $R_u=\underline{R}$. 
Let $\phi\in \cH^{1, 1}_{w_+}$ be the potential and $V_\phi$ be the log volume ratio. Then $V_\phi$ satisfies the equation in $\Delta$-weak sense,
\[
\Delta_\phi V_\phi=g^{i\bar j}_\phi R_{i\bar j}-\underline{R}. 
\]
Next we consider $u\in W^{1, 2}$ such that for any $\psi\in C^\infty$,
\[
\int_M g^{i\bar j}_\phi u_i \psi_{\bar j} \omega_\phi^n=\int_M \psi (-g^{i\bar j}_\phi R_{i\bar j}+\underline{R})\omega_\phi^n
\]
In other words, $u$ is a $W^{1, 2}$ solution of the equation \[\Delta_\phi u=g^{i\bar j}_\phi R_{i\bar j}-\underline{R}.\] Then standard elliptic regularity theory implies $u\in C^{\alpha}$. 
Note that $u$ is also a $\Delta$-weak solution, as well as $V_\phi$. 
By the uniqueness result, we know that $V_\phi-u$ is a constant. Then $V_\phi$ is also a $W^{1, 2}$ weak solution and hence it is $C^\alpha$. It then follows that $\phi$ is smooth.  We refer to Section 6 for the details of a priori estimates which are independent of $\epsilon$. 
\end{proof}

It would be an interesting problem to obtain such a regularity result locally' namely, drop the smallness assumption, which is considered in Section 3. But it seems that there are some subtle difficulties. Compactness does play a role in our arguments and we will discuss this below. 
First it is not very hard to extend our $W^{2, 2}$ estimates on compact K\"ahler manifolds to the following interior $W^{2, 2}_{loc}$ estimate.

\begin{thm}\label{local1}Let $u_{i\bar j}$ be an $L^\infty$ metric such that $\lambda^{-1}\delta_{ij}\leq u_{i\bar j}\leq \lambda\delta_{ij}$.
 Consider the following linear equation over the unit ball $B$ for $v\in W^{1, 2}$ and $f\in L^2$, 
\[
\p_i\left(\det(u_{i\bar j})u^{i\bar j}\p_{\bar j} v\right)=f\det(u_{i\bar j})
\]
then we have $v\in W^{2, 2}_{loc}$. Moreover it solves the equation in the strong sense that
\[
u^{i\bar j}\p^2_{i\bar j} v=f. 
\]
\end{thm}

\begin{proof}The proof is purely local. By standard elliptic theory, $v\in W^{1, 2}\cap C^\alpha$ for some $\alpha>0$. Consider the following approximation of $u$ by $u_\beta\in C^\infty$, such that 
\[
(2\lambda)^{-1}\delta_{ij}\leq (u_\beta)_{i\bar j}\leq 2\lambda\delta_{ij}
\]
holds for on $B_R$ ($R$ is a fixed number in $(7/8, 1)$ and we will take $R\rightarrow  1$ finally) and $u_\beta\rightarrow u\in W^{2, p}_{loc}$ for $p$ sufficiently large. Let $f_\beta\rightarrow f$ in $L^2$ and $f_\beta\in C^\infty$.
We solve the equation over $B_R$, for each $\beta$,
\[
\p_i\left(\det((u_\beta)_{i\bar j})u^{i\bar j}_\beta\p_{\bar j} v_\beta\right)=f_\beta\det((u_\beta)_{i\bar j})
\]
with the boundary restriction $v_\beta=v$ on $\p B_R$, where the boundary condition is understood in the sense that $v_\beta-v\in W^{1, 2}_0(B_R)$. Then $v_\beta\in C^\infty(B_R).$ We want to show that $v_\beta\in W^{2, 2}_{loc}(B_R)$ with uniform bounds.  Choose a smooth cut-off function $\eta\in C^{\infty}(B_R)$, supported in $B_R$, and $\eta=1$ in $B_{3/4}$. Then the standard computation gives
\[
\int_{B} (u^{i\bar j}_\beta (v_\beta)_{i\bar j})\eta^2 (u^{p\bar q}_\beta (v_\beta)_{p\bar q})\det((u_\beta)_{i\bar j})=-\int_B \eta^2 f_\beta\det((u_\beta)_{i\bar j}) u^{p\bar q}_\beta (v_\beta)_{p\bar q}
\]
A direct computation gives
\[
\int_{B_{3/4}} u^{i\bar j}_\beta u^{p\bar q}_\beta (v_\beta)_{, i\bar q} (v_\beta)_{, p\bar j} \det((u_\beta)_{i\bar j})\leq C(\|f_k\|_{L^2}, \lambda, |\nabla \eta|)
\]
Hence it gives the following,
\[
\int_{B_{3/4}}|(v_\beta)_{i\bar j}|^2\leq C,
\]
for some uniformly bounded constant $C$. 
This gives that $v_\beta \in W^{2, 2}(B_{1/2})$ uniformly. Clearly the same argument applies to $B_r\subset B_R$ for any $r<R$. It follows that $v_\beta\in W^{2, 2}_{loc}(B_R)\cap W^{1, 2}(B_R)$ (after bypassing to a subsequence if necessary). In particular we have
\[
u^{i\bar j}_\beta (v_\beta)_{i\bar j}=f_\beta. 
\]
Now let $\beta\rightarrow \infty$,  $v_\beta$ converges to a solution of the equation over $B_R$
\[
\p_i\left(\det(u_{i\bar j})u^{i\bar j}\p_{\bar j} w\right)=f\det(u_{i\bar j})
\]
and $w=v$ on $\p B_R$. In particular $w\in W^{2, 2}_{loc}(B_R)$ and $w$ solves $u^{i\bar j}w_{i\bar j}=f$ in strong sense. By the uniqueness we know that $w=v$, hence $v\in W^{2, 2}_{loc}(B_R)$. Finally we can let $R\rightarrow 1$. This completes the proof.\end{proof}

We would like to ask whether the local regularity result for constant scalar curvature equation holds,

\begin{prob}\label{local} Let $u_{i\bar j}$ be an $L^\infty$ K\"ahler metric such that $(u_{i\bar j})$ gives a $\Delta$-weak solution of $R_u=\underline{R}$. Can we prove that $u$ is indeed smooth locally?
\end{prob}

The local $W^{2, 2}$ estimates seems not to be sufficient. Indeed we need a version of uniqueness result as the following, 

\begin{prob}\label{weakuniqueness2}Let $u_{i\bar j}$ be an $L^\infty$ K\"ahler metric over $B$ and $v$ be a $\Delta$-weak solution of $\Delta_u v=0$ in the sense that
\[
\int_B v u^{i\bar j}\p^2_{i\bar j}\phi \det(u_{i\bar j})=0
\]
for any $\phi\in C^\infty_0(B)$. Suppose $v\in L^\infty$ and $v=0$ on $\p B$. Is that true  $v=0$?
\end{prob}

This is a local version the uniqueness result in Corollary \ref{weakuniqueness}. But it seems that there are some tricky problems. One point is that $v$ is the $\Delta$-weak solution of the harmonic equation $\Delta_u v=0$ with only $L^\infty$ coefficients. On compact manifolds, we have shown that the only harmonic solutions are constants using the global $W^{2, 2}$ estimates. But this argument cannot be extended directly to the local setting. For example local harmonic functions exist in abundance. Problem \ref{weakuniqueness2} asks whether one can prescribe a reasonable boundary condition (zero boundary value) to show that a $\Delta$-weak harmonic function with $L^\infty$ coefficients is zero. One main trouble is that the definition of $\Delta$-weak harmonic function is not up to boundary (the test functions have compact supports).  Another difficulty is that for a measurable function (say an $L^\infty$ function), it is tricky to prescribe a sensible ``zero boundary value". Hence Problem \ref{local} remains interesting to be studied, including how to propose boundary value, for example. Of course if we assume that $u$ is smooth near boundary, one can get interior local smoothness for a weak CSCK. But such an assumption seems to be too artificial and we skip the details.

Next we discuss extremal metrics and modified $K$-energy briefly. We have the following.
\begin{thm} An invariant $L^\infty$ K\"ahler metric which minimizes modified $K$-energy is a smooth extremal K\"ahler metric. 
\end{thm}

Since the arguments and discussions are very similar to the $K$-energy case, we shall keep the arguments brief. First we recall the modified $\cK$-energy, defined by Chen-Guan \cite{GC}(see also Chen-Tian\cite{CT2}),
\[
\delta \cK_X=-\int_M \delta\phi (R_\phi-\underline{R}-\theta_X(\phi))\omega_\phi^n,
\]
where $X$ is the extremal vector field and it is unique up to holomorphic transformation by a result of Futaki-Mabuchi(\cite{FM}), and $\theta_X(\phi)$ is the potential of $X$ with respect to $\omega_\phi$, such that $\nabla^{1, 0}_{\omega_\phi} \theta_X(\phi)=X$. We assume that $\omega_{\phi}$ is invariant under the action of $Im(X)$; such metrics are called invariant metrics and for invariant metrics, $\theta_X(\phi)$ is a well-defined real-valued function, such that
\[
L_X\omega_\phi=\sqrt{-1}\p\bar\p \theta_X(\phi). 
\] 
Now for only $C^{1, 1}$ potentials, the modified $K$-energy $\cK_X$ has a similar Chen's formula and it makes sense to talk about minimizers. If a minimizer is $L^\infty$, it satisfies the equation, in the $\Delta$-weak sense,
\[
R_\phi-\underline{R}=\theta_X(\phi). 
\]
For a $C^{1, 1}$ potential, it is evident that $\theta_X(\phi)$ is uniform bounded (assuming a normalized condition); indeed by a result of X. Zhu \cite{Zhu}  (Section 5) $\theta_X(\phi)$ is uniformly bounded for all invariant metrics (even though we do not really need this result with the $L^\infty$ assumption on $\omega_\phi$). In other words, we consider an $L^\infty$ invariant K\"ahler metric, such that its volume ratio $V_\phi$ satisfies the following equation in $\Delta$-weak sense,
\[
\Delta_\phi V_\phi=g^{i\bar j}_\phi R_{i\bar j}-\underline{R}-\theta_X(\phi). 
\]
We can run the arguments above to consider the $W^{1, 2}$ weak solution for
\[
\Delta_\phi u=g^{i\bar j}_\phi R_{i\bar j}-\underline{R}-\theta_X(\phi). 
\]
And we know that $u\in W^{2, 2}$ and it is also a $\Delta$-weak solution. By uniqueness we know $V_\phi=u+\text{const}$. This is sufficient to prove that $V_\phi\in W^{2, 2}\cap C^\alpha$ and this completes the proof.  

\section{Twisted CSCK and a priori estimate}

In this section we prove Theorem \ref{estimate}.

\begin{proof}[Proof of Theorem \ref{estimate}] We assume $\psi$ has average zero, $\int_M\psi =0$. Then $\Delta\psi<\infty$ implies that $\psi$ is uniformly bounded (note that $0<n+\Delta\psi$). Now we want to show that there exists a positive constant $C$,
\[
C^{-1}\omega\leq \omega_\psi.
\]
The result clearly holds for a uniformly bounded constant $\delta=\delta(\omega, n)$, by the openness result at $t=0$ of the second author \cite{Zeng} and Hashimoto \cite{Hashimoto}. Now we assume $t\geq \delta$. We rewrite the equation as
\begin{equation}
-t g^{i\bar j}_\psi \p^2_{i\bar j}\left(\log\frac{\omega_\psi^n}{\omega^n}+A\psi\right)=t\underline{R}+(1-t+At)\tr_{\omega_\psi}\omega-n(1-t+At)+t\tr_{\omega_\psi}(Ric),
\end{equation}
for a sufficiently large constant $A$. Suppose $\log\frac{\omega_\psi^n}{\omega^n}+A\psi$ achieves its minimum at a point $p$. Then at $p$, we have
\[
-t g^{i\bar j}_\psi \p^2_{i\bar j}\left(\log\frac{\omega_\psi^n}{\omega^n}+A\psi\right)(p)\leq 0.
\]
It follows that, at $p$, 
\[
t\underline{R}+(1-t+At)\tr_{\omega_\psi}\omega-n(1-t+At)+t\tr_{\omega_\psi}(Ric)\leq 0
\]
This implies that $\l_i\geq C^{-1}$ at $p$, where $\l_i$ are eigenvalues of $\omega_\psi$ with respect to $\omega$. It follows in particular that $\omega_\psi^n\geq C^{-1}\omega^n$ at $p$. Since $\psi$ is uniformly bounded and $\log\frac{\omega_\psi^n}{\omega^n}+A\psi$ achieves minimum at $p$, we get that $\omega_\psi^n\geq C^{-1}\omega^n$ holds. Together with the assumption $\Delta\psi<\infty$, we obtain $C^{-1}\omega\leq \omega_\psi.$ Once we obtain $C^{-1}\omega\leq \omega_\psi\leq C\omega$, it is  straightforward to obtain all higher order estimates, using various well-known elliptic theory, as in Section 2. We shall skip the argument. 
\end{proof}

In general it would be a very hard problem to obtain the bound on $\sup \Delta \psi$ for CSCK or twisted CSCK equation. On the other hand, such estimates holds for the complex Monge-Ampere equation by the seminal work of Yau \cite{Yau} and Aubin \cite{Aubin}.

\end{document}